\documentclass[12pt]{amsart}
\usepackage{amsmath}
\usepackage{amssymb}
\usepackage{amstext}
\usepackage{amscd}
\usepackage[francais, english]{babel}
\usepackage[matrix,arrow,ps]{xy}

\newtheorem{teo}{Theorem}[section]
\newtheorem{prop}[teo]{Proposition}
\newtheorem{lem}[teo]{Lemma}

\newcommand{\GL}{{\rm GL}}

\newcommand{\CC}{{\mathbb C}}
\newcommand{\RR}{{\mathbb R}}
\newcommand{\ZZ}{{\mathbb Z}}
\newcommand{\QQ}{{\mathbb Q}}
\newcommand{\NN}{{\mathbb N}}

\newcommand{\PP}{{\mathbb P}}

\newcommand{\AAA}{{\mathbb A}}

\newcommand{\lto}{\longrightarrow}

\def\FP{\mathfrak{p}}

\newcommand{\cO}{{\mathcal O}}

\newcommand{\cF}{\mathcal{F}}

\newcommand{\cS}{\mathcal{S}}
\newcommand{\cL}{{\mathcal L}}

\newcommand{\ol}{\overline}

\newcommand{\wt}{\widetilde}

\title{Holomorphic curves in compact Shimura varieties.}
\author{Emmanuel Ullmo, Andrei Yafaev}
\address{Ullmo: IHES 35 Route de Chartres, 91440 Bures-sur-Yvette, France}
\email{ullmo@ihes.fr}
 \address{Yafaev: UCL, Department of Mathematics, Gower street, WC1E 6BT, London, UK}
\email{yafaev@math.ucl.ac.uk}

\begin{document}
\maketitle

\begin{abstract}
We prove a hyperbolic analogue of the Bloch-Ochiai theorem about the Zariski closure 
of holomorphic curves in abelian varieties. 
\end{abstract}

\selectlanguage{francais}
\begin{abstract}
On d\'emontre un analogue hyperbolique du th\'eor\`eme de Bloch-Ochiai sur l'adh\'erence de Zariski
d'une courbe holomorphe dans une vari\'et\'e ab\'elienne. 
\end{abstract}
\selectlanguage{english}
\tableofcontents

\section{Introduction.}

The following theorem of Bloch-Ochiai (see Chapter 9, theorem 3.9.19 of \cite{Kobayashi}) is proved using Nevanlinna theory.

\begin{teo} [Bloch-Ochiai] \label{bloch}
Let $A$ be an abelian variety and $f \colon \CC \lto A$ be a non-constant holomorphic map. Then the Zariski closure of $f(\CC)$ is a translate of an abelian subvariety.
\end{teo}

In this paper we formulate and prove an analogue of this theorem for a certain type of locally symmetric varieties, namely the compact Shimura varieties.

For notations and facts about Shimura varieties and weakly special subvarieties, we refer to \cite{UY} and references therein. 
Recall that
any hermitian symmetric domain $X$, admits a realisation $X \subset \CC^n$ (with $n = \dim(X)$)
as a bounded symmetric domain. See \cite{Mok}, Chapter 4 for details.

Recall also that given an arithmetic lattice $\Gamma \subset {\rm Aut}(X)^+$, such that the quotient $\Gamma \backslash X$ is compact, there exists a fundamental domain $\cF$ for the action of $\Gamma$ on $X$ which is an open subset of $X$ such that $\ol{\cF}$ is compact.
For a bounded hermitian symmetric domain $X \subset \CC^n$, we denote by $\partial X$ the boundary of $X$, i.e. 
$\partial X = \ol{X} \backslash X$ where $\ol{X}$ denotes the topological closure of $X$ in $\CC^n$.

For notions of Shimura data, Shimura varieties and their weakly special subvarieties we refer to \cite{DeDe}, \cite{UY} and references contained therein.
We just recall that weakly special subvarieties are defined in terms of Shimura subdata, but as shown in \cite{MoMo}, they are exactly the totally geodesic subvarieties of $\Gamma \backslash X$ and terms `weakly special' and `totally geodesic' are used in the literature interchangeably.

Let $(G,X)$ be a Shimura datum with $G$ anisotropic over $\QQ$, let  $X^+$ be a connected component of $X$ and 
$K$ a compact open subgroup of $G(\AAA_f)$. As above, $X^+ \subset \CC^n$ is a bounded symmetric domain.

We let $\Gamma$ be the intersection of $K$
with the stabiliser of $X^+$ in $G(\QQ)$.  Then $\Gamma$ is an arithmetic congruence group acting on $X^+$.

Then $S = \Gamma \backslash X^+$  is compact. 
Let $\pi \colon X^+ \lto \Gamma \backslash X^+$ be the quotient map.

\begin{teo} \label{main_thm}
Let $f \colon \CC \lto \CC^n$ be a holomorphic map such that $C=f(\CC)\cap X^+$ is non-empty. 

The components of the Zariski closure $Zar(\pi(C))$ of $\pi(C)$ are weakly special subvarities of $S$.
\end{teo}

This result is partly inspired by the following so-called hyperbolic Ax-Lindemann theorem whose slightly different but equivalent formulation is proven in 
\cite{UY}, Th\'eor\`eme 1.3 in the co-compact case and in \cite{KUY} for all Shimura varietes.

\begin{teo} \label{ax_lin}
With the notations of theorem \ref{main_thm}, let $Y$ be an algebraic subset of $X^+$
(i.e. a component of an intersection of an algebraic subset of $\CC^n$ with $X^+$).

The components of the Zariski closure of $\pi(Y)$ are weakly special subvarieties of $S$.
\end{teo}

The proof of \ref{main_thm} relies on the theory of o-minimality and the Pila-Wilkie theorem and is inspired by the 
proof of the hyperbolic Ax-Lindemann theorem in the co-compact case as in \cite{UY}.
The proof also uses in an essential way the hyperbolic Ax-Lindemann theorem itself and the results of \cite{U}.

An analogous question in the context of abelian varieties has been investigated  in \cite{UY1}. In that paper we have not been able to re-prove Bloch-Ochiai theorem 
using o-minimal techniques. We however obtained a result analogous to \ref{main_thm} in the abelian context for certain sets definable in the usual o-minimal structures. Our result in \cite{UY1} is in some ways more general than the Bloch-Ochiai theorem.
It is surprising and interesting that the obstructions to prove the Bloch-Ochiai theorem using o-minimality do not occur in the hyperbolic case we consider here, however additional serious difficulties arise which we overcome in section 3. 

The strategy of the proof is as follows. We start by decomposing $f^{-1}(f(\CC)\cap X^+)$ as a union of connected components $U_i \subset \CC$.
For a given $i$ we prove that for some $R_i>0$, it is in fact enough to prove the conclusion for $\pi \circ f(U_i \cap B(0,R_i))$ where $B(0,R_i)$ is the open ball centered at the origin of radius $R_i$. This is done in section \ref{prem}.

We now set $C_i = f(U_i \cap B(0,R_i))$. Section 3 is the technical heart of the proof. 
The analytic curve $C_i$ in $X^+$ is definable in the o-minimal structure $\RR_{an}$
(here $\CC^n$ is identified with $\RR^{2n}$). 
For o-minimality, related notions and results we refer to \cite{VDD}.
We fix a fundamental domain $\cF$ for the action of $\Gamma$ on $X^+$.

We let $V_i$ be the Zariski closure of $\pi(C_i)$ and $\wt{V_i}$ be $\pi^{-1}(V_i) \cap \cF$.
We associate to $C_i$ a certain definable (in $\RR_{an}$) set $\Sigma \subset G(\RR)$ and show
that $\Sigma \cdot C_i \subset \wt{V}_i$. 
The main technical work is to prove that $\Sigma$ contains \emph{a lot} of points of $G(\QQ)$ of height up to $T$. Pila-Wilkie theorem then allows us to conclude that $\Sigma$ contains a positive dimensional semi-algebraic subset $W$ and the hyperbolic Ax-Lindemann theorem allows us to conclude that $V_i$ contains a Zariski dense set of weakly special subvarieties.
Using results from \cite{U} and some additional arguments, we conclude the proof of theorem \ref{main_thm}.

\section*{Acknowledgements.}

The second author is very grateful to the IHES for hospitality during his visits in May and September 2016.
The second author gratefully acknowledges 
financial support of the ERC, Project 511343.

\section{Preliminaries.} \label{prem}

Keep notations as in Theorem \ref{main_thm}. For simplicity of notation, we write $X$ for $X^+$. 
Let
$$
f^{-1}(f(\CC)\cap X) = \coprod_{i\in I} U_i
$$ 
be the decomposition of  $f^{-1}(f(\CC)\cap X)$ into connected components.
By definition of the $U_i$, for each $i$, we have
$$
\ol{f(U_i)}\cap \partial X \not= \emptyset.
$$

For $R_i > 0$ large enough, we have
$$
\ol{f(B(0,R_i)\cap U_i)} \cap \partial X \not= \emptyset
$$
(where $B(0,R_i)$ is the open ball of radius $R_i$ centered at the origin).
For each $i$, we fix an $R_i$ with this property.
\begin{prop} \label{olivia}
We have
$$
Zar(\pi\circ f(U_i)) = Zar(\pi \circ f(B(0,R_i)\cap U_i).
$$
\end{prop}

\begin{proof}
One inclusion is obvious.
Write $Zar(\pi\circ f(B(0,R_i)\cap U_i)) \subset \PP^m$ for some $m$ and let $s \in H^0(\PP^m, {\mathcal O}(l))$ for 
$l\geq 1$ such that $s$ is zero on $\pi\circ f(B(0,R_i)\cap U_i)$.
Then the function $s \circ f \circ \pi \colon U_i \lto \CC$ is zero on
$B(0,R_i)\cap U_i)$. Since $U_i$ is connected, by analytic continuation, the function $s \circ f \circ \pi$ is zero on $U_i$.
It follows that $s$ is zero on $\pi\circ f(U_i)$. This proves the other inclusion.
\end{proof}

In this paper we will prove the following:

\begin{teo} \label{main_thm1}
The Zariski closure of $\pi \circ f(B(0,R_i)\cap U_i)$ contains a Zariski dense subset of
weakly special subvarieties.
\end{teo}

Let $V=V_i$ be the Zariski closure of $\pi \circ f(B(0,R_i)\cap U_i)$.
The theorem \ref{main_thm1} will be deduced form the following:

\begin{teo} \label{semialg}
There exists a positive dimensional semialgebraic set $W$ in $G(\RR)$
such that 
$$W\cdot f(B(0,R_i)\cap U_i) \subset  \pi^{-1}(V).$$
\end{teo}

To deduce \ref{main_thm1} from \ref{semialg}, let $P \in f(B(0,R_i)\cap U_i)$.
For the notion of algebraic subset of $X$, we refer to Appendix B of \cite{KUY}.
There exists a complex algebraic subset $Y_P \subset \pi^{-1}(V)$ such that
$W \cdot P \subset Y_P$ (see \cite{KUY}, Lemma B.3). By Ax-Lindemann theorem \ref{ax_lin}, the Zariski closure of
$\pi(Y_P) \subset V$ is weakly special. Therefore, through each point of $\pi f(B(0,R_i)\cap U_i)$
there passes a weakly special subvariety and hence $V$ contains a dense set of weakly special subvarieties.

We will now prove that theorem \ref{main_thm} follows from theorem \ref{main_thm1}.
Let $V$ now be a component of the Zariski closure of $\pi(f(\CC)\cap X)$. 
By theorem \ref{main_thm1}, $V$ contains a Zariski dense set of weakly special subvarieties.

If $V$ is a special subvariety, then we are done. Assume that $V$ is not special.
By the main theorem of \cite{U}, there exists a special subvariety $S' \subset S$ containing $V$ and such
that $S' = S_1 \times S_2$ ($S_i$ special and positive dimensional) and such that
$$
V = S_1 \times V'
$$
where $V'$ is a subvariety of $S_2$.

There exists a sub-Shimura datum $(G',X')$ of $(G,X)$ and a decomposition
$$
(G'^{ad}, X'^{ad}) = (G_1,X_1) \times (G_2, X_2)
$$
such that $S_1 = \Gamma_1 \backslash X_1$ and $S_2 = \Gamma_2 \backslash X_2$ (where as usual we omit the superscript $+$)
 and $\Gamma_1$ and $\Gamma_2$ are suitable arithmetic lattices in $G_1(\QQ)^+$ and $G_2(\QQ)^+$.

Let $\FP_1 \cong \CC^{r_1}$ and $\FP_2 \cong\CC^{r_2}$ be the holomorphic tangent spaces to $X_1$ and $X_2$.
Then $\FP_1 \times \FP_2$ is a subspace of the holomorphic tangent space $\FP \cong \CC^n$ to $X$.
Let
$$
f^{-1}(f(\CC) \cap X) = \coprod_{i \in I} U_i.
$$
be as before, the connected component decomposition.
There exists a $U_i$ such that the restriction, $f \colon U_i \lto \CC^n$
factors through $\CC^{r_1} \times \CC^{r_2}$. By analytic continuation $f \colon \CC \lto \CC^n$
factors through $\CC^{r_1} \times \CC^{r_2}$.

Let $f_1$ and $f_2$ be the holomorphic functions from $\CC$ to $\CC^{r_1}$ and $\CC^{r_2}$ respectively
such that $f = (f_1, f_2)$.

Similarly, write
$$
f_2^{-1}(f_2(\CC) \cap X_2) = \coprod_{j \in J} V_j.
$$
the connected component decomposition.

For any $i \in I$, there exists an $j \in J$ such that $U_i \subset V_j$ because 
both  $f_2(U_i)$ are contained in  $X_2$ respectively.

It follows that for any $i \in I$, there exists $j \in J$,
$$
Zar(\pi_2\circ f_2 (U_i)) = Zar(\pi_2 \circ f_2 (V_j)).
$$
Note that $V$ is the Zariski closure of  the union of the
$\pi_2\circ f_2 (U_i)$. Therefore $V$ is the Zariski closure of 
$$
\bigcup_{i\in I} Zar(\pi_2\circ f_2 (U_i)).
$$
By theorem \ref{main_thm1}, $Zar(\pi_2\circ f_2 (U_i)) = Zar(\pi_2 \circ f_2 (V_j))$ contains a Zariski dense set of weakly special subvarieties.

It follows that $V'$ contains a Zariski dense set of weakly special subvarieties of $S_2$. An inductive argument finishes the proof
theorem \ref{main_thm} assuming theorem \ref{main_thm1}.

\section{Counting lattice elements.}

In this section we show that $f(U_i)$ (as in the previous section) in $X$ intersects ``exponentially many'' (in a suitable sense) $\Gamma$-translates of a fixed fundamental domain. This section constitutes the technical heart of the paper.

Recall the following notations from \cite{UY}. 
Let $X$ be a connected Hermitian symmetric domain (as usual we omit the superscript $+$), realised as a bounded  symmetric domain in  some $\CC^n$. We let $C$ to be $f(B(0,R_i)\cap U_i)$ with $R_i$ and $U_i$ as in the previous section. 

Let $\Gamma$ be a cocompact arithmetic lattice in the group $G$ of holomorphic isometries of $X$.
For a point $x_0 \in X$, we let $\cF$ be a fundamental domain 
for the action of $\Gamma$ on $X$ such that $x_0 \in \cF$. We assume that $\cF$ is an open connected set such that $\ol{\cF}$ is compact.
The set 
$$
\cS_{\cF} = \{ \gamma \in \Gamma : \gamma \cF \cap \cF \not= \emptyset \}
$$
is finite and generates $\Gamma$. 

The ``word metric'' $l\colon \Gamma \lto \NN$ with respect to $\cS_{\cF}$ is defined as follows $l(1)=0$ and for $\gamma \not= 1$, $l(\gamma)$ is the minimal number of elements of $\cS_{\cF}$ needed to write $\gamma$ as their product.

We also let $K(Z,W)$ be the Bergmann kernel on $X$ and we let
$$
\omega = \sqrt{-1} \partial \ol{\partial} K(Z,Z)
$$
be the associated K\"ahler form.
We refer to \cite{Mok} 4.1 for details on this.

We define the following functions:

$$
N_C(n) = |\{ \gamma \in \Gamma : \dim(\gamma \cF \cap C) = 1, l(\gamma )\leq n \}|
$$

and 

$$
N'_C(n) = |\{ \gamma \in \Gamma : \dim(\gamma \cF \cap C) = 1, l(\gamma ) =  n \}|.
$$

The main result of this section is the following theorem:

\begin{teo} \label{counting}
There is a positive constant $c$ such that for all $n\gg 0$, we have
$$
N_C(n) \geq e^{cn}.
$$
\end{teo}

Let $b$ be a point of the boundary of $\ol{C} \cap \partial X$ and a neighbourhood $V_b$ of $b$ such that
$\ol{C} \cap \partial X \cap V_b$ is a real analytic curve.

We parametrise $\ol{C} \cap \partial X \cap V_b$ as follows. For $0<\alpha, \beta <2\pi$, let
$\Delta_{\alpha, \beta}$ be the sector of the unit disc $\Delta$ defined as follows:
$$
\Delta_{\alpha, \beta} = \{ z = re^{i\theta} :  0 \leq r  \leq 1, \alpha \leq \theta \leq \beta  \}.
$$
Let $C_{\alpha, \beta}$ be the subset of
$\partial \Delta_{\alpha, \beta}$ defined as
$$
C_{\alpha,\beta} = \{ z = e^{i\theta} : \alpha \leq \theta \leq \beta \}.
$$
We can find $\alpha$, $\beta$ and a real analytic map $\psi$ from a neignbourhood of $\Delta_{\alpha,\beta}$ to $\CC^n$
such that 
$\psi (\Delta_{\alpha, \beta}) \subset \ol{C} \cap X$
and $\psi(C_{\alpha, \beta}) \subset \ol{C} \cap \partial X$.

Let $\Delta$ be the open unit disk.
We let $\omega_{\Delta}$ be the usual Poincar\'e $(1,1)$-form on  $\Delta$
($\omega_\Delta = \sqrt{-1}\frac{dz \wedge d\ol{z}}{(1-|z|^2)^2}$).
By lemma 2.8 of \cite{UY}, there exists a smooth $(1,1)$-form $\eta$ on $\Delta_{\alpha, \beta}$ such that
$$
\psi^* \omega = s \omega_{\Delta} + \eta
$$
for some integer $s>0$.

Let $\gamma \in \Gamma$ be such that $\dim(\gamma \cF \cap C) = 1$ and $\gamma \cF \cap C \subset \psi(\Delta_{\alpha, \beta})$, then
\begin{equation} \label{volume}
\int_{\gamma \cF \cap C} \omega = s \int_{\psi^{-1}(\gamma \cF \cap C)} \omega_{\Delta} + 
\int_{\psi^{-1}(\gamma \cF \cap C)} \eta
\end{equation}

\begin{prop} \label{laurent}
There exists a constant $B$ such that for any $\gamma \in \Gamma$ such that
$\dim(\gamma \cF \cap C) = 1$, we have
$$
\int_{\gamma \cF \cap C} \omega \leq B.
$$ 
\end{prop}
\begin{proof}
We consider the compact dual $X_c$ of $X$ which is a projective algebraic variety. Let $\cL$ be the dual of the canonical line bundle endowed with a $G(\CC)$-invariant metric $||.||_{FS}$. We let $\omega_{FS}$ the associated $(1,1)$-form: 
$\omega_{FS} = c_1(\cL, ||.||_{FS})$.

By Harish-Chandra embedding theorem, there is a biholomorphism $\lambda$ from $\FP \cong \CC^n$ to an open dense subset of $X_c$. 
For details, see Theorem 1, section 5.2 of \cite{Mok}.

Let $\gamma \in \Gamma$ be such that $\gamma \cF \cap C \not= \emptyset$. Since $\omega$ is 
$\Gamma$-invariant, we have 
$$
\int_{\gamma \cF \cap C} \omega = \int_{\gamma^{-1}(\gamma \cF \cap C)} \omega.
$$
On the compact set $\ol{\cF}$, the two forms $\omega$ and $\lambda^*(\omega_{FS})$ are positive holomorphic forms, therefore there is a constant $B_1$ such that on $\cF$, we have
$$
\omega \leq B_1 \lambda^* (\omega_{FS}).
$$

We have
$$
\int_{\gamma \cF \cap C} \omega \leq B_1 \int_{\gamma^{-1}(\gamma \cF \cap C)} \lambda^*(\omega_{FS}).
$$
Furthermore,
$$
\int_{\gamma^{-1}(\gamma \cF \cap C)} \lambda^*(\omega_{FS}) \leq \int_{\gamma^{-1}\lambda(C)}\omega_{FS}.
$$
The conclusion of proposition \ref{laurent} follows from the following lemma that will be proven in the following section. :
\begin{lem} \label{luca}
There is a constant $B_2$ such that for all $\gamma \in \Gamma$, we have
$$
\int_{\gamma^{-1}\lambda(C)}\omega_{FS} \leq B_2.
$$
\end{lem}
\end{proof}

\subsection{Proof of lemma \ref{luca}.}

The volume of the analytic curve $\lambda(C)$ is defined as
$$
{\rm Vol}(\lambda(C)) = \int_{\lambda(C)} \omega_{FS}. 
$$

Let $\PP^{n\vee}$ be the dual projective space, the set of hyperplanes in $\PP^n$.
Let $dJ$ be the invariant volume element on $\PP^{n\vee}$ normalised to have total mass one.

By Generalised Crofton's formula (see \cite{Alex} and references therein), we have
$$
{\rm Vol}(\gamma^{-1}\lambda(C)) = \alpha \int_{\PP^{n\vee}} n_{\gamma^{-1}\lambda(C)}(J) dJ.
$$
where $\alpha$ is a uniformisation constant and $n_{\gamma^{-1}\lambda(C)}(J)$ is the number of points (counted with multiplicity) 
of the
intersection of $\gamma^{-1}\lambda(C)$ with $J$.
Note that the function $n_{\gamma^{-1}\lambda(C)}(J)$ is a function defined on the open subset of
$\PP^{n\vee}$ consisting of hyperplanes $J$ such that $\gamma^{-1}\lambda(C)$ is not contained $J$.
The complement of this open set is of measure zero, therefore the integral is well defined.

\begin{lem}
Let $J$ be a hyperplane in $\PP^n$ and $\gamma \in \Gamma$.
There exists a hyperplane $J'$ such that 
$$
n_{\gamma^{-1}\lambda(C)}(J) = n_{\lambda(C)}(J').
$$ 
\end{lem}
\begin{proof}
Recall that $\cL$ is $\Gamma$-invariant and $\cL$ is very ample i.e. $\cL = \cO(1)|_{X_c}$.
Write $s$ the section of $\cL$ such that 
$J \cap X_c = {\rm div}(s)$.
Let $s' = \gamma^* s$. Then $s'$ is a restriction of a section of $\cO(1)$ corresponding to some hyperplane $J'$ and we thus have
$$
\gamma (J \cap X_c) = J' \cap X_c.
$$
Therefore
$$
\lambda(C) \cap \gamma (J \cap X_c) = \lambda(C) \cap (J' \cap X_c).
$$

We also have 
$$
\lambda(C) \cap \gamma (J \cap X_c) = \gamma ( \gamma^{-1} \lambda(C) \cap J ).
$$

We conclude using the fact that 
$$
|\gamma ( \gamma^{-1} \lambda(C) \cap J )| = n_{\gamma^{-1}\lambda(C)}(J).
$$
\end{proof}

We finish by proving a general lemma:

\begin{lem}
Let $f \colon \CC \lto \PP^n(\CC)$ be a holomorphic map.
Let $R >0$. 
There exists a constant $\Theta = \Theta(R,f)$  such that for any hyperplane $H$ of $\PP^n(\CC)$
such that $f(\CC)$ is not contained in $H$,
$$
| \{ f(B(0,R)) \cap H  \}| \leq \Theta.
$$
\end{lem}
\begin{proof}
A reference for notions of Nevanlinna-Cartan theory is \cite{Kobayashi}, Chapter 3, Section B. We use notations from this reference.

Let $N(R,f,H)$ be the counting function associated to $f,R$ and $H$.
Let $\alpha_1, \dots , \alpha_t \in B(0,R)$ be complex numbers such that $f(\alpha_i) \in H$.

Let $\nu(f,\alpha_i,H)$ be the multiplicity of $f$ in $H$ at $\alpha_i$.
We have
$$
N(R,f,H) = \sum_{i=1}^t \nu(f,\alpha_i,H) \log(\frac{R}{|\alpha_i |}).
$$

Therefore

$$
N(2R,f,H) \geq \sum_{i=1}^t \nu(f,\alpha_i,H) \log(\frac{2 R}{|\alpha_i |}).
$$

We have $\log(\frac{2R}{|\alpha_i |}) \geq \log(2)$, therefore
a bound on $N(2R,f,H)$ implies a bound on $\sum_{i=1}^t \nu(f,\alpha_i,H) = | \{ f(B(0,R)) \cap H  \}|$.
 It is hence enough to bound $N(2R,f,H)$.
 
 The first main theorem of Cartan-Nevanlinna theory (\cite{Kobayashi}, 3.B.16), we have
 $$
 N(2R,f,H) \leq T(2R,f) + c 
 $$
where $c$ is a uniform constant and $T(2R,f)$ is the order function defined in \cite{Kobayashi}, 3.B.2.

Since $T(2R,f)$ does not depend on $H$, this concludes the proof.
\end{proof}

\subsection{End of proof of theorem \ref{counting}.}

As $\eta$ is smooth on $\Delta_{\alpha, \beta}$, the integral $\int_{\psi^{-1}(\gamma \cF \cap C)} \eta$ is bounded independently of $\gamma$.
Equation \ref{volume} and lemma \ref{laurent} imply that $\int_{\psi^{-1}(\gamma \cF \cap C)} \omega_{\Delta}$ is bounded
by a constant $B'$, independent of 
$\gamma$.

Recall the following lemma (Lemma 2.1) from \cite{UY}. Note that this lemma
is proved in \cite{UY} for $C$ algebraic but the algebraicity assumption is not used,
the statement and proof remain the same in our situation. In fact the proof is a combination of
some general facts about hermitian symmetric domains and word metrics. 

\begin{lem} \label{length_compare}
There exist positive constants $\lambda_1$ and $\lambda_2$ and $D$ such that for all $z \in \Delta_{\alpha,\beta}$ with $z \in \psi^{-1}(\gamma \cF \cap C)$,
$$
\lambda_1 l(\gamma) \leq - \log(1 - z \ol{z}) \leq \lambda_2 l(\gamma) + D.
$$
\end{lem}

We now follow the end of section 2 of \cite{UY}.

For $n>0$, let
$$
I_n = \{ z\in \Delta_{\alpha, \beta}, e^{-(n+1)} \leq 1 - |z|^2 \leq e^{-n}\}.
$$

The hypebolic volume of $I_n$ satisfies

$$
{\rm Vol}(I_n) \geq \delta_1 e^n
$$
where $\delta_1$ is a positive constant.

The set $I_n$ is covered by the $\phi^{-1}(\gamma \cF \cap C)$.
For each $n$ large enough and for all $z \in I_n$, by lemma \ref{length_compare}, there exists a $\gamma$ such that $\psi(z)\in \gamma \cF$ with $\gamma$ satisfying
$$
c_1 n \leq l(\gamma) \leq c_2 n
$$
with uniform constants $c_1$ and $c_2$.

On the other hand, for all $z \in \Delta_{\alpha, \beta}$, such that 
$\psi(z)\in \gamma \cF$ for some $\gamma \in \Gamma$,
$$
{\rm Vol}(\psi^{-1}(\gamma \cF \cap C)) \leq B'.
$$

Therefore, by the computation of ${\rm Vol}(I_n)$ above, there exists a $\delta_1 > 0$ such that
$$
\sum_{c_1n \leq k \leq c_2 n} N'_C(k) \geq \delta_1 e^n.
$$

This finishes the proof of theorem \ref{counting}.

\section{A definable set and application of Pila-Wilkie theorem.}

In this section we prove theorem \ref{semialg} and hence our main theorem. 
We follow section 5 of \cite{UY} with appropriate modifications.

Let $U$ be as before a connected component of $f^{-1}(f(\CC)\cap X)$ and $R$ such that
$f(U\cap B(0,R)) \cap \partial X \not= \emptyset$. 
Note that $C = f(B(0,R)\cap U)$ is definable in $\RR_{an}$. Let $\cF$ be as in the previous section. Recall (see \cite{UY}, Proposition 4.2) that
$\pi$ restricted to $\cF$ is definable in $\RR_{an}$.

Consider
$$
\Sigma(C) = \{ g \in G(\RR) : \dim(gC \cap \pi^{-1}(V) \cap \cF) = 1 \}.
$$
The set $\Sigma(C)$ is definable in $\RR_{an}$.

We prove the following.

\begin{lem}  \label{props}
\begin{enumerate}
\item For all $g \in \Sigma(C)$, $g C \subset \pi^{-1}(V)$.
\item Define
$$
\Sigma'(C) = \{ g \in G(\RR) : g^{-1}\cF \cap C \not= \emptyset \}.
$$
Then
$$
\Sigma(C)\cap \Gamma = \Sigma'(C)\cap \Gamma.
$$
\end{enumerate}
\end{lem}
\begin{proof}
Let $g \in \Sigma(C)$, then
$$
g C \cap \cF \subset \pi^{-1}(V).
$$
By analytic continuation, this implies that $g C \subset \pi^{-1}(V)$.

The proof of (2) is exactly identical to the proof of  \cite{UY}, Lemma 5.2 and relies on the fact that $\pi^{-1}(V)$ is $\Gamma$-invariant.
\end{proof}

From previous lemma and theorem \ref{counting}, we obtain the following.

\begin{lem}
Let
$$
N_{\Sigma(C)}(n) =  |\{ \gamma \in \Gamma \cap \Sigma(C) : l(\gamma) \leq n \}|.
$$
For all $n$ large enough, 
$$
N_{\Sigma(C)}(n) \geq e^{cn}.
$$
\end{lem}
The height $H(\gamma)$ of an element $\gamma$ of $\Gamma$ is defined by viewing $\Gamma$ as a subgroup of some $\GL_m(\ZZ)$ and taking the maximum of the absolute values of the entries.
If $l(\gamma) \leq n$, then $H(\gamma) \leq (mA)^n$ where $A$ is the maximum of heights of elements of $\cS_{\cF}$. 

Let now 
$$
\Theta(\Sigma(C), T) = \{ g \in G(\QQ)\cap \Sigma(C) : H(g) \leq T \}
$$
and 
$$
N(\Sigma(C), T) = |\Theta(\Sigma(C), T)|
$$

\begin{lem} \label{count}
$$
N(\Sigma(C),T) \geq T^{c_1}.
$$
\end{lem}

We now appeal to the Pila-Wilkie theorem (see \cite{PW}, Theorem 1.8).

For a definable (in some o-minimal structure) subset $\Theta\subset \RR^n$, we define $\Theta^{alg}$ to be 
the union of all positive dimensional semi-algebraic subsets contained in $\Theta$. We define $\Theta^{tr}$ to be
$\Theta \backslash \Theta^{alg}$.

\begin{teo}[Pila-Wilkie]
Let $\Theta$ be a subset of $\RR^{n}$ definable in an o-minimal structure.
Let $\epsilon > 0$. There exists a constant $c = c(\Theta,\epsilon)$ such that
for any $T \geq 0$,
$$
|\{ x \in \Theta^{tr} \cap \QQ^n : H(x) \leq T \}| \geq c T^{\epsilon}.
$$
\end{teo}

In view of lemma \ref{count}, by Pila-Wilkie theorem, there exists a positive dimensional semi-algebraic subset 
$W \subset \Sigma(C)$ and by (1) of lemma \ref{props}, we have
$W \cdot C \subset \pi^{-1}(V)$.
This finishes the proof of theorem \ref{semialg}.


\begin{thebibliography}{99}

\bibitem{Alex} H. Alexander, {\it Volumes of varieties in projective space and in Grassmannians.} Transactions of the AMS, Vol {\bf 289}, 1974.

\bibitem{DeDe} P. Deligne, {\it Vari\'et\'es de Shimura: interpr\'etation modulaire, et techniques de construction de mod\'eles canoniques.} Automorphic forms, representations and L-functions (Proc. Sympos. Pure Math., Oregon State Univ., Corvallis, Ore., 1977), Part 2, pp. 247-289, Proc. Sympos. Pure Math., XXXIII, Amer. Math. Soc., Providence, R.I., 1979. 

\bibitem{Kobayashi} S. Kobayashi, {\it Hyperbolic Complex Spaces.}
Vol. {\bf 318}. A Series of Comprehensive studies in Mathematics, Springer 1998.

\bibitem{KUY} B. Klingler, E. Ullmo, A. Yafaev ,{\it The hyperbolic Ax-Lindemann-Weierstrass conjecture.} Publ. Math. Inst. Hautes Etudes Sci. {\bf 123} (2016), 333-360. 

\bibitem{Mok} N. Mok, {\it Metric Rigidity Theorems on Hermitian Locally Symmetric Manifolds.}
Series in Pure Math. 6. World Scientific, 1989.

\bibitem{MoMo} B. Moonen,	{\it Linearity properties of Shimura varieties. I.} J. Algebraic Geom. 7 (1998), 539-567.



\bibitem{PW} J. Pila, A. Wilkie, {\it The rational points of a definable set.} Duke Math. Journal, {\bf 133(3)}, 591-616, 2006.

\bibitem{U} E. Ullmo, {\it Applications du theor\`eme d'Ax-Lindemann hyperbolique.}
Compositio Mathematica, Vol {\bf 150}, Issue 02, 175-190, 2014.

\bibitem{UY} E. Ullmo, A. Yafaev, {\it Hyperbolic Ax-Lindemann theorem in the co-compact case.}
Duke Math. Journal, Vol {\bf 163}, Number 2, 433-463, 2014.

\bibitem{UY1} E. Ullmo, A. Yafaev, {\it o-minimal flows on abelian varieties.} 
Preprint, 2016.

\bibitem{VDD} L. Van Den Dries, {\it Tame topology and o-minimal structures.}
LMS Lecture Notes Series, {\bf 248}, 1998.



\end{thebibliography}
\end{document}